\documentclass[12pt,a4paper]{amsart}
\usepackage{amsmath,amsfonts,amssymb}
\usepackage{setspace}

\usepackage{longtable}
\usepackage[mathscr]{eucal}
\usepackage{amsmath, amsthm}
\usepackage{mathrsfs}
\usepackage{amsbsy}
\usepackage{dsfont}
\usepackage{bbm}
\usepackage{wasysym}
\usepackage{stmaryrd}
\usepackage{url}

\theoremstyle{plain}
\usepackage{color}
\usepackage[breaklinks]{hyperref}

\input xypic
\xyoption{all}

\makeindex
\makeglossary

\begin{document}
\baselineskip = 16pt
\onehalfspacing

\newcommand \ZZ {{\mathbb Z}}
\newcommand \NN {{\mathbb N}}
\newcommand \RR {{\mathbb R}}
\newcommand \PR {{\mathbb P}}
\newcommand \AF {{\mathbb A}}
\newcommand \GG {{\mathbb G}}
\newcommand \QQ {{\mathbb Q}}
\newcommand \CC {{\mathbb C}}
\newcommand \bcA {{\mathscr A}}
\newcommand \bcC {{\mathscr C}}
\newcommand \bcD {{\mathscr D}}
\newcommand \bcF {{\mathscr F}}
\newcommand \bcG {{\mathscr G}}
\newcommand \bcH {{\mathscr H}}
\newcommand \bcM {{\mathscr M}}
\newcommand \bcI {{\mathscr I}}
\newcommand \bcJ {{\mathscr J}}
\newcommand \bcK {{\mathscr K}}
\newcommand \bcL {{\mathscr L}}
\newcommand \bcO {{\mathscr O}}
\newcommand \bcP {{\mathscr P}}
\newcommand \bcQ {{\mathscr Q}}
\newcommand \bcR {{\mathscr R}}
\newcommand \bcS {{\mathscr S}}
\newcommand \bcV {{\mathscr V}}
\newcommand \bcU {{\mathscr U}}
\newcommand \bcW {{\mathscr W}}
\newcommand \bcX {{\mathscr X}}
\newcommand \bcY {{\mathscr Y}}
\newcommand \bcZ {{\mathscr Z}}
\newcommand \goa {{\mathfrak a}}
\newcommand \gob {{\mathfrak b}}
\newcommand \goc {{\mathfrak c}}
\newcommand \gom {{\mathfrak m}}
\newcommand \gon {{\mathfrak n}}
\newcommand \gop {{\mathfrak p}}
\newcommand \goq {{\mathfrak q}}
\newcommand \goQ {{\mathfrak Q}}
\newcommand \goP {{\mathfrak P}}
\newcommand \goM {{\mathfrak M}}
\newcommand \goN {{\mathfrak N}}
\newcommand \uno {{\mathbbm 1}}
\newcommand \Le {{\mathbbm L}}
\newcommand \Spec {{\rm {Spec}}}
\newcommand \Gr {{\rm {Gr}}}
\newcommand \Pic {{\rm {Pic}}}
\newcommand \Jac {{{J}}}
\newcommand \Alb {{\rm {Alb}}}
\newcommand \Corr {{Corr}}
\newcommand \Chow {{\mathscr C}}
\newcommand \Sym {{\rm {Sym}}}
\newcommand \Prym {{\rm {Prym}}}
\newcommand \cha {{\rm {char}}}
\newcommand \eff {{\rm {eff}}}
\newcommand \tr {{\rm {tr}}}
\newcommand \Tr {{\rm {Tr}}}
\newcommand \pr {{\rm {pr}}}
\newcommand \ev {{\it {ev}}}
\newcommand \cl {{\rm {cl}}}
\newcommand \interior {{\rm {Int}}}
\newcommand \sep {{\rm {sep}}}
\newcommand \td {{\rm {tdeg}}}
\newcommand \alg {{\rm {alg}}}
\newcommand \im {{\rm im}}
\newcommand \gr {{\rm {gr}}}
\newcommand \op {{\rm op}}
\newcommand \Hom {{\rm Hom}}
\newcommand \Hilb {{\rm Hilb}}
\newcommand \Sch {{\mathscr S\! }{\it ch}}
\newcommand \cHilb {{\mathscr H\! }{\it ilb}}
\newcommand \cHom {{\mathscr H\! }{\it om}}
\newcommand \colim {{{\rm colim}\, }} % colimit
\newcommand \End {{\rm {End}}}
\newcommand \coker {{\rm {coker}}}
\newcommand \id {{\rm {id}}}
\newcommand \van {{\rm {van}}}
\newcommand \spc {{\rm {sp}}}
\newcommand \Ob {{\rm Ob}}
\newcommand \Aut {{\rm Aut}}
\newcommand \cor {{\rm {cor}}}
\newcommand \Cor {{\it {Corr}}}
\newcommand \res {{\rm {res}}}
\newcommand \red {{\rm{red}}}
\newcommand \Gal {{\rm {Gal}}}
\newcommand \PGL {{\rm {PGL}}}
\newcommand \Bl {{\rm {Bl}}}
\newcommand \Sing {{\rm {Sing}}}
\newcommand \spn {{\rm {span}}}
\newcommand \Nm {{\rm {Nm}}}
\newcommand \inv {{\rm {inv}}}
\newcommand \codim {{\rm {codim}}}
\newcommand \Div{{\rm{Div}}}
\newcommand \CH{{\rm{CH}}}
\newcommand \sg {{\Sigma }}
\newcommand \DM {{\sf DM}}
\newcommand \Gm {{{\mathbb G}_{\rm m}}}
\newcommand \tame {\rm {tame }}
\newcommand \znak {{\natural }}
\newcommand \lra {\longrightarrow}
\newcommand \hra {\hookrightarrow}
\newcommand \rra {\rightrightarrows}
\newcommand \ord {{\rm {ord}}}
\newcommand \Rat {{\mathscr Rat}}
\newcommand \rd {{\rm {red}}}
\newcommand \bSpec {{\bf {Spec}}}
\newcommand \Proj {{\rm {Proj}}}
\newcommand \pdiv {{\rm {div}}}
\newcommand \wt {\widetilde }
\newcommand \ac {\acute }
\newcommand \ch {\check }
\newcommand \ol {\overline }
\newcommand \Th {\Theta}
\newcommand \cAb {{\mathscr A\! }{\it b}}

\newenvironment{pf}{\par\noindent{\em Proof}.}{\hfill\framebox(6,6)
\par\medskip}

\newtheorem{theorem}[subsection]{Theorem}
\newtheorem{conjecture}[subsection]{Conjecture}
\newtheorem{proposition}[subsection]{Proposition}
\newtheorem{lemma}[subsection]{Lemma}
\newtheorem{remark}[subsection]{Remark}
\newtheorem{remarks}[subsection]{Remarks}
\newtheorem{definition}[subsection]{Definition}
\newtheorem{corollary}[subsection]{Corollary}
\newtheorem{example}[subsection]{Example}
\newtheorem{examples}[subsection]{examples}

\title{Bloch's conjecture on surfaces with $p_g=q=0, K^2=9$}
\author{Kalyan Banerjee}
\email{kalyan.ba@srmap.edu.in}
\maketitle
\begin{abstract}
    In this paper, we prove that Bloch's conjecture holds for all smooth complex projective surfaces with $p_g=q=0$ and $K^2=9$.
\end{abstract}
\section{Introduction}
In \cite{M}, Mumford proved that if the geometric genus of a smooth projective surface is greater than zero, then the Chow group of zero-cycles on the surface is not finite-dimensional, in the sense that the natural map from the symmetric powers of the surface to the Chow group is not surjective. This means that the Albanese kernel is non-trivial and enormous and cannot be parametrized by an algebraic variety. The converse question is whether, for a surface with geometric genus zero and irregularity zero, the Chow group of zero-cycles is supported at one point. This conjecture was originally formulated by Spencer Bloch. It is known, by \cite{BKL}, for surfaces not of general type with geometric genus equal to zero. Subsequently, the conjecture was verified for some examples of surfaces of general type with geometric genus zero in \cite{Ba}, \cite{B}, \cite{IM}, \cite{GP}, \cite{PW} %\cite{Gul1}, %\cite{Gul2}, 
\cite{V}, \cite{VC}. 

After posting this preprint on arXiv, the author learned of a preprint by Professor Vladimir Guletskii claiming Bloch's conjecture for surfaces with $p_g=0$ in general. The present preprint is also of independent interest, as its argument is independent of that in Professor Guletskii's preprint \cite{Gul}.

 The aim of this manuscript is to prove Bloch's conjecture for surfaces of general type with geometric genus zero, irregularity zero, and $K^2=9$. Such surfaces were discovered in \cite{M1} and extensively studied in \cite{CK}, \cite{K}, and \cite{PY}.
 
 The main theorem of this paper is as follows.

\begin{theorem}
Consider a smooth, projective, minimal surface of general type $X$ with $p_g=q=0$ and $K^2=9$ over $\CC$. Then Bloch's conjecture holds for $X$.
\end{theorem}

We now introduce the notation used in this text.

Let $X$ be a smooth projective surface with geometric genus $p_g=0$ and irregularity $q=0$ over the field of complex numbers.

Let $K_X$ denote the canonical bundle on $X$.

Let $A_0(X)$ denote the group of zero-cycles of degree zero modulo rational equivalence.

We denote by $C_t$ a smooth curve in the linear system $|nK_X|$, where $n>>5$, and by $D_{\bar\eta}$ a very general member of the linear system $|5K_X|$, of genus 136.

Let $g$ be the genus of a curve in the linear system $|nK_X|$.

We denote by $F_s$ the general fiber of the map $X^{g-9n-1}\to A_0(X)$, which factors through the Jacobian fibration, denoted by $\bcJ\times J(D_{\bar\eta})$.

Let $U$ denote the Zariski open subset of $X^{g-9n-1}$ consisting of all points $(s_1,\cdots, s_{g-9n-1})$ such that $s_i\neq s_j$ for $i\neq j$.

Here, $\Theta$ denotes the map from $X^{g-9n-1}$ to $A_0(X)$.

\section{Proof of the main theorem}

\begin{theorem}
    Let $X$ be a smooth projective minimal surface of general type with $p_g=q=0$ and $K^2=9$ over $\CC$. Then Bloch's conjecture holds for $X$.
\end{theorem}

\begin{proof}
Let $X$ be a surface of general type with $p_g=q=0$ over the complex numbers. Then, by the Bogomolov--Miyaoka--Yau inequality \cite{Bog}, \cite{Mi}, \cite{Y}, $K_X^2\leq 9$.
Moreover, by the result of Bombieri et al. \cite{Bo}, we know that $5K_X$ is very ample. 

Consider $nK_X$, where $n>>5$. Then, by the adjunction formula, we have 
$$n(n+1)K_X^2=2g-2$$
where $g$ is the genus of a curve in the linear system $|nK_X|$.

Assume that $K_X^2=9$. We then need to compute the dimension of the linear system $|nK_X|$. The corresponding plurigenus is 
$$9n(n-1)/2+\chi(0)=9n(n-1)/2+1$$
This follows from \cite[Theorem 2]{Bo}.

Thus, the dimension of the linear system $|nK_X|$ is $9n(n-1)/2$. By the adjunction formula, this is $g-1-9n$.

Let us prove the following lemma.

\begin{lemma}
Consider the product variety $X^{g-1-9n}$. Let $U$ be the Zariski open subset of $X^{g-1-9n}$ containing points $(s_1,\cdots,s_{g-1-9n})$ such that $s_i\neq s_j$ for $i\neq j$. Then there is a unique curve $C_t$ in $|nK_X|$ passing through the points $s_1,\cdots,s_{g-1-9n}$.
\end{lemma}

\begin{proof}
In fact, the curves containing a single point $s_i$ form a codimension-one subvariety of $|nK_X|$. If $s_i$ and $s_j$ are distinct, then we claim that these two codimension-one subvarieties intersect properly; that is, their intersection has codimension $2$. This follows from Bertini's theorem. The curves $C_t$ containing $s_i$ give rise to a hyperplane in $|nK_X|$. Thus, if we consider the incidence variety 
$$V:=\{(s_i, t)\mid s_i\in C_t\}\subset X\times |nK_X|$$
the fiber over $s_i$ is the projective space defined by a hyperplane. The projective space defined by $s_j$ is another hyperplane in $|nK_X|$. Since $s_i\neq s_j$, these two hyperplanes are distinct; that is, they give distinct closed points in the dual projective space. The intersection of these two hyperplanes contains both $s_i$ and $s_j$. Hence, the corresponding hyperplanes are distinct, and their intersection has codimension 2. Proceeding in this way, we obtain a unique curve containing all the points $s_i$ such that $s_i\neq s_j$ and $1\leq i\leq g-1-9n$.
 \end{proof}

Therefore, consider the zero-cycle 
$$[\sum_{i=1}^{g-1-9n}s_i- (g-1-9n)s_0]\;,$$ 
which is supported on $J(C_t)$ and lies in the image of $C_t^{g-1-9n}$, where $s_0$ is a fixed point on $X$. 

Now consider the very ample divisor $5K_X$. Suppose that $nK_X$ is such that $|5K_X|$ is embedded in $|nK_X|$. The dimension of $|5K_X|$ is 90 and the dimension of $|nK_X|$ is $\frac{9n(n-1)}{2}$.

Among these $g-1-9n$ points, take $90$ points that are in general position with respect to the linear system $|5K_X|$. Let us take $n$ copies of $|5K_X|$ that are disjoint in $|nK_X|$.

Therefore, from the $g-1-9n$ points, we can remove the $90n$ points lying on certain unique curves $D_1,D_2,\ldots,D_n$ belonging to $|5K_X|$.

These $90n$ points, chosen from the $g-1-9n$ points, are supported on $J(D_1)\times J(D_n)$, where each $D_i$ is a curve in $|5K_X|$. 

\begin{lemma}
The image of $\prod_i J(D_i)$ in $A_0(X)$ is supported on $J(D_{\bar\eta})$, where $D_{\bar\eta}$ is the geometric generic curve in the linear system of $|5K_X|$.
\end{lemma}

\begin{proof}
Consider the image of $J(D_1)\times \cdots\times J(D_n)$ in $A_0(X)$. Denote this subgroup by $\Gamma$. This subgroup is finite-dimensional in Roitman's sense; that is, it is dominated by an algebraic variety. Now, $J(D_i)$ has dimension $136$. We need to prove that the image of $\prod_i J(D_i)$ is supported on the Jacobian of a single curve.

Consider the variety $\prod_i J(D_i)$. Let $C$ be an ample hyperplane section of this variety. By the Lefschetz hyperplane theorem, it follows that $J(C)$ surjects onto the product $\prod_i J(D_i)$.

Let $g_0$ be the dimension of $J(C)$, which is at least $136n$, since $J(C)$ surjects onto the product. Since $5K_X$ is very ample, the Jacobian $J(D_i)$ is simple. This is a consequence of monodromy of Lefschetz pencils in the linear system $|5K_X|$ and the Picard-Lefschetz formula. For a proof, please see \cite{Voi}[Proposition 1.4]. Hence, $J(C)\to J(D_i)$ is surjective, and its kernel is an abelian variety in $J(C)$, say $B_i$, such that $J(C)/B$ is isogenous to $J(D_i)$. Hence, $\dim J(C)/B\leq 136$. Thus, we can replace $J(D_i)$ by $A_i=J(C)/B_i$. Then $\prod_i J(D_i)$ is isogenous to $\prod_i A_i$. Now, $J(D_i)\cong J(D_j)\cong J(D_{\bar\eta})$, so $\prod_i J(D_i)$ is isomorphic (non-canonically) to $J(D_{\bar\eta})^n$, which is isogenous to $\prod_i A_i$.

Now observe that, if we consider the intersection number $C_t.D_i$, it is the same as $C_t.D_j$ when $D_i$ and $D_j$ are general members of the linear system $|5K_X|$. 
Therefore, considering $90n$ general points lying in the intersection $C_t.D_i$ for all $i=1,2,\cdots,n$, we have $90$ points in $D_{\bar\eta}$, provided that the points in $X^{g-1-9n}$ are in general position; that is, they lie in the complement of the union of Zariski closed subsets $Z_i$. Here, $Z_i$ is defined as the subset of $(s_1,\cdots,s_{g-1-9n})$ such that, for any $90$ points among $s_1,\cdots,s_{g-1-9n}$, the curve $D\in |5K_X|$ passing through these 90 points is not a very general curve in the linear system $|5K_X|$. This can always be achieved as follows.

Let $Z$ be the subset of $|5K_X|$ corresponding to the curves in the linear system that are not very general. That is, $Z=\cup_i Z_i$, where $Z_i$ is Zariski closed in $|5K_X|$. Then consider the incidence subvariety 
$$V_1:\{(s_1,\cdots,s_{g-1-9n},t):s_{1},\cdots,s_{90}\in D_t\}\subset X^{g-1-9n}\times |5K_X|$$
Consider the second projection $V\to |5K_X|$ and the preimage of $Z$ under this projection. The map $V\to |5K_X|$ has a fiber of dimension $90+2(g-1-9n-90)=2(g-1-9n)-90$. Thus, the dimension of $V$ is $2(g-1-9n)$. The dimension of the pullback of $Z_i$, for all $i$, is much smaller than that of $V_1$, and hence the same is true of its projection to $X^{g-1-9n}$, since the map from $V_1$ to $X^{g-1-9n}$ is finite.

Similarly, if we take another combination of points $s_{i1},s_{i2},\cdots,s_{i90}$ and define $V_i$ in the same way, then the dimension of the pullback of $Z_i$ to $X^{g-1-9n}$ is much smaller than $2g-2-18n$. Taking the union of all these $Z_i$ in $X^{g-1-9n}$ and then considering its complement, we obtain very general points satisfying the statement above, namely, that the curve $D_i$ is very general.

Now consider the pushforwards $j_{i*}:J(D_i)\to A_0(X)$ and $j_{\bar\eta*}:J(D_{\bar\eta})\to A_0(X)$. They commute by Lemma 4.4.4 in \cite{BTh}. Hence, all the cycles in $J(D_i)$, for all $i$, can be transferred to $J(D_{\bar\eta})$. Let us give the details.

Fix one $D_i$ and an isomorphism with $D_{\bar\eta}$. Denote this isomorphism by $\phi$. Then we need to prove that 
$$j_{\bar\eta*}(\phi(z))-j_{D_{i*}}(z)=0$$
in $A_0(X)$. Consider the subgroup of $A_0(X)$ generated by cycles of the form $$j_{\bar\eta*}(\phi(z))-j_{D_{i*}}(z)$$
Denote this subgroup by $\Gamma$. This is finite dimensional because it is dominated by $J(D_{\bar\eta})\times J(D_i)$. Now consider the map 
$$j_*:=(\phi(z),z)\mapsto j_{\bar\eta*}(\phi(z))-j_{D_{i*}}(z)$$
The kernel of this homomorphism is a countable union of translates of an abelian subvariety $A$ of $J(D_{\bar\eta})\times J(D_i)$. Therefore, the image of $A$ under the projection to $J(D_i)$ is either $0$ or all of $J(D_i)$. Similarly, the projection of $A$ to $J(D_{\bar\eta})$ is either $0$ or all of $J(D_{\bar\eta})$. Let these two projections be $A_{\bar\eta}$ and $A_{D_i}$. If $A_{D_i}=0$, its image under $\phi$ is $0$. However, $(\phi(z),z)$ lies in $J(D_{\bar\eta})\times J(D_i)$, and if it lies in the kernel, then, when $z=0$, $\phi(z)=0$; in this case, the kernel is countable. Otherwise, it is all of $ J(D_i)$. Hence, in that case, we have
$$j_{\bar\eta*}(\phi(z))-j_{D_{i*}}(z)=0\;.$$
Now, if the kernel is countable, then there is a countable subset of $J(D_i)$, say $\{z_i\}$, such that each $z_i$ satisfies 
$$j_{\bar\eta*}(\phi(z_i))-j_{D_{i*}}(z_i)=0\;.$$
Now consider the torsion points in $J(D_i)$. Since $A_0(X)$ is torsion-free, the torsion points are among these $z_i$'s. Since the torsion points are dense in $J(D_i)$, we have that $\ker(j_*)$ is Zariski dense in the image of $J(D_i)$ under $\phi$ in $J(D_{\bar\eta})\times J(D_i)$. Therefore, we have a Zariski dense subset of $\Sym^{136} D_{\bar\eta}\times Sym^{136} D_i$, where $136=g(D_i)=g(D_{\bar\eta})$. Denote it by $Z$. Let $Z_i$ be its projection to $\Sym^{136} D_i$, which is also Zariski dense. 

\begin{lemma}
Let $\pi: X \to Y$ be a flat morphism of finite presentation between algebraic schemes, where $Y$ is irreducible. Let $Z \subseteq X$ be a constructible subset. If the set of points 
\[
Y_{\text{dense}} = \{ y \in Y \mid Z_y \text{ is Zariski dense in } X_y \}
\]
is Zariski dense in $Y$, then $Z$ is Zariski dense in $X$.
\end{lemma}

\begin{proof}
We proceed by contraposition. Suppose that $Z$ is not Zariski dense in $X$. By definition, this means that the Zariski closure of $Z$ in $X$, denoted by $\overline{Z}$, is a proper closed subscheme of $X$. 

Consequently, the complement 
\[
U = X \setminus \overline{Z}
\]
is a non-empty, Zariski open subset of $X$. Since $Z \subseteq \overline{Z}$, it follows directly from the construction that $Z$ and $U$ are disjoint:
\begin{equation}\label{eq:disjoint}
Z \cap U = \emptyset
\end{equation}

Next, we consider the image of this open set under our fibration, $\pi(U) \subseteq Y$. Because $\pi$ is a flat morphism of finite presentation, it is an open map. Since $U$ is a non-empty open subset of $X$, its image $V = \pi(U)$ must be a non-empty open subset of $Y$.

Since the base space $Y$ is assumed to be irreducible, every non-empty open subset of $Y$ is automatically Zariski dense in $Y$. Thus, $V$ is a dense open subset of $Y$. 

Now, let $y \in V$ be an arbitrary point of this open subset. By the definition of the image, the fiber of $U$ over $y$, given by $U_y = U \cap X_y$, must be non-empty. Furthermore, because $U$ is open in $X$, its restriction to the fiber, $U_y$, is a non-empty Zariski open subset of the fiber $X_y$.

From our disjointness relation in equation \eqref{eq:disjoint}, intersecting both sides with the fiber $X_y$ yields:
\[
(Z \cap X_y) \cap (U \cap X_y) = \emptyset \implies Z_y \cap U_y = \emptyset
\]
Because $U_y$ is a non-empty Zariski open subset of $X_y$ containing no points of $Z_y$, the fiber $Z_y$ cannot be Zariski dense in $X_y$. 

Since this holds for every point $y \in V$, the set of points where the fiber is dense is contained in the complement of $V$:
\[
Y_{\text{dense}} \subseteq Y \setminus V
\]
Because $V$ is a non-empty open set, its complement $Y \setminus V$ is a proper closed subset of $Y$. Therefore, $Y_{\text{dense}}$ cannot be Zariski dense in $Y$. 

By contraposition, if $Y_{\text{dense}}$ is Zariski dense in $Y$, then $Z$ must be Zariski dense in $X$.
\end{proof}
Applying this lemma to the fibration $Sym^{136} D_U\to U$, it follows that there exists $Z_U$ that is Zariski dense in $\Sym^{136} D_U$, where $D_U\to U$ is the fibration of smooth projective curves over the Zariski open subset $U$ of $|5K_X|$. The fiber of $Z_U$ at $t$ is $Z_t$.  The dimension of $\Sym^{136}D_U$ is $226$. The components of the Zariski dense subset $Z$ have dimension $90$. These $90$-dimensional components are images of the projective space $|5K_X|$.

Consider  the map $\pi:Sym^{136} D_U\to Sym^{136} X$ and its image, we have a Zariski dense subset of the image of $\pi$ as above.  Since $|5K_X|$ has dimension 90, we have a Zariski dense subset of $\pi(Sym^{136}D_U)$ that is a countable union of images of projective spaces. These images contain many rational curves. Since $X^{136}$ is Kobayashi hyperbolic, such a rational curve cannot exist and collapses to a point, contradicting the fact that the image is positive-dimensional. Therefore, the images of the components of $Z$, say $Z'$, are of dimension 0. The components of $Z$ are contained in the countably many fibers of $\pi$. This means that countably many fibers are dense in $\Sym{136}D_U$. Therefore, the generic fiber of $Sym^{136} D_U\to U$ intersects one of these special fibers, which is impossible, as the Jacobian of the generic fiber would then contain a rational curve, which is not possible. Therefore the kernel of the map 

$$(\phi(z),z)\to j_{\bar\eta*}(\phi(z))-j_{D*}(z)$$
is all of $J(D_i)$ and the isomorphism $\phi$ commutes with the push-forward.

\end{proof}

Thus, we have a map from $X^{g-1-9n}\to A_0(X)$ given by 

$$(s_1,\cdots,s_{g-1-9n})\to [\sum_{i=1}^{g-1-9n}s_i- (g-1-9n)s_0]\;,$$
which factors through $\bcJ\times \prod_j J(D_j)$, where $\bcJ$ is the Jacobian fibration of curves in $|nK_X|$ and $D_j$ is a general member of the linear system $|5K_X|$.  
%Note that,  very general curves $D_1,D_2,\cdots, D_n$ are isomorphic as schemes, and hence so is that for Jacobians (because we are working over an uncountable algebraically closed field). This isomorphism is highly noncanonical. That is $J(D_i)=J(D_{\bar \eta})$ where ${\bar \eta}$ is the geometric generic point of $|5K_X|$. 

Thus, by the preceding discussion, the map $X^{g-1-9n}\to A_0(X)$ factors through $X^{g-1-9n}\to \bcJ\times J(D_{\bar\eta})$.

 The dimension of $X^{g-1-9n}$ is $2g-2-18n$, and the dimension of the image of $X^{g-1-9n}$ in $\bcJ\times J(C)$ is at most $2g-2-108n+136$, where $k=\dim(J(D_{\bar\eta}))\leq 136$. Thus, the fiber dimension is $90n-136>0$.

Now consider the ample line bundle $5K_X$ and pull it back to $X^{g-1-9n}$. This pulled-back ample divisor intersects the fiber of the above map, and hence we see that elements of the fiber are supported on $X^{g-2-9n}\times D$. Iterating this process as long as $n-k>0$, we obtain the following.

\begin{lemma}
The elements of the fiber are supported on 
$$X^{g-k-1-9n}\times D^{136}\;.$$
\end{lemma}

\begin{proof}
Indeed, consider the divisor
$\sum_i \pi_i^{-1}(5K_{X})$ on the product $X^{g-1-9n}$. This divisor is ample, so it intersects $F_s$, and we find that there exist points in $F_s$ contained in $D\times X^{g-2-9n}$, where $D$ belongs to the linear system $|5K_{X}|$. Then consider the elements of $F_s$ of the form $(c,s_1,\cdots,s_{g-2-9n})$, where $c$ belongs to $D$. Considering the map from $X^{g-2-9n}$ to $A_0(X)$ given by
$$(s_1,\cdots,s_{g-2-9n})\mapsto [ \sum_j s_j-(g-9n-2)s_0]\;,$$
we see that this map factors through the product of the Jacobian fibration and the Jacobian of $D_i$, and the map from $X^{g-2-9n}$ to $\bcJ\times J(D_{\bar\eta})$ has positive-dimensional fibers, since $90n-136$ is large.

Thus, iterating this process, we find that there exist elements of $F_s$ that are supported on $D^k\times X^{g-k-1-9n}$, where $k$ is some natural number depending on $g$. The genus of the curve $D$ is $136$, since $K_X^2=9$. We can choose the very ample line bundle $nK_X$ so that $g$ is very large and $k$ is larger than the genus of $D$. Thus, we have $z_s$ supported on $D^{136}\times X^{g-k-1-9n}$.

\end{proof}

Hence, we have $\Theta(U)=\Theta(X^{i_0})$, where $i_0=135+g-k-9n$, which is strictly less than $g-1-9n$, since the genus of $D$ is strictly less than $k$. Here, a little care must be taken: we have $k>135$ and $$g-k-1-9n>0.$$
Thus, we have the following inequality:
$$g-9n>k+1>136.$$
Now $g=9n(n+1)/2+1$; substituting this into the inequality above, we obtain
$$9n(n-1)/2>k>135,$$
that is,
$$9n(n-1)/2>k>135.$$
This can be achieved by taking $n$ large.

At the same time, to keep the dimension of the fiber positive, we need $90n-136>0$.

Thus, we have shown that, for any $(s_1,\cdots,s_{g-1-9n})$ in $U\subset X^{g-1-9n}$, $z_s$ is rationally equivalent to a cycle on $X^{i_0}$. Now, the proof follows from the following lemmas.

\begin{lemma}\cite{Voi}(Fact 3.3)
Let $X$ be a smooth projective complex variety, and let $U$ be the complement of a countable union of Zariski-closed subsets of $X$. Then any zero-cycle on $X$ is rationally equivalent to a zero-cycle on $U$.
\end{lemma}

\begin{proof}
Let $z$ be a zero-cycle on $X$ and let $C$ be a curve containing the support of $z$. In addition, there exists a curve that intersects $U$. This is because the curves of a fixed degree that contain the support of $z$ and are contained in the complement of $U$ form a proper closed subvariety of the Hilbert scheme of curves on $X$ of that fixed degree. 

Let $C$ be a curve that intersects $U$ non-trivially and contains the support of $z$. Then we write $z=z_1-z_2$ in $\Pic(C)$ so that the degrees of $z_1$ and $z_2$ are greater than $2g+1$, where $g$ is the genus of the curve $C$. Then $z_1$ and $z_2$ are very ample divisors on $C$. Therefore, the linear systems $|z_1|$ and $|z_2|$ are base-point-free. Consider any $z_1'\in |z_1|$ and $z_2'\in |z_2|$ such that $z_1'$ and $z_2'$ do not contain the points of $C\setminus (C\cap U)$. Indeed, if every $z_i'\in |z_i|$ contained the points of $C\setminus (C\cap U)$, then the linear systems would have base points. Thus, for such $z_1'$ and $z_2'$, we have 
$$z=z_1'-z_2'$$
and hence the lemma.
\end{proof}

If $k>>136$, then the above argument shows that the image of $X^{g-9n-1}$ is supported on $X^{g-9n-k+135}$; hence, $A_0(X)$ is finite-dimensional by the following lemma.

We now prove the following lemma by induction.
\begin{lemma}\cite{Voi}
  We have $\Theta(X^i)=\Theta(X^{i_0})$ for all $m\geq i_0$.  
\end{lemma}
 
\begin{proof}
Suppose that $$\Theta(X^k)=\Theta(X^{i_0})$$ for $k\geq i_0$. We then have to prove that $$\Theta(X^{k+1})=\Theta(X^{i_0}).$$
Any element in $\Theta(X^{k+1})$ can be written as $$(s_1+\cdots+s_{i_0}+s)-(i_0+1)s_0.$$

Now let $k-i_0=m$. Then $i_0+1=k-m+1$. Since $k-m<k$, we have $k-m+1\leq k$, and hence $i_0+1\leq k$. Therefore, the above cycle is
supported on $X^k$ and hence on $X^{i_0}$. Thus, we have $\Theta(X^{i_0})=\Theta(X^k)$ for all $k$ greater than or equal to $i_0$.
\end{proof}
\end{proof}

\subsection{Caution about this proof}
This proof does not work for a K3 surface $X$ with $p_g=1$ and $q=0$. Let us give some details. First of all, $K_X=0$. Thus, we have $L.K_X=0$ for any sufficiently large line bundle $L$. Let $L$ be a sufficiently large line bundle on $X$. Then we have $L^2=2g-2$, and the dimension of $|L|$ is $g$. Now consider $X^g$. As above, any point of $X^g$ whose coordinates are all distinct lies on a unique curve $C_t\in |L|$. Then there is a map
$$\Theta: X^g\to A_0(X)$$
as above, and it factors through $\bcJ$. Here, $\bcJ$ has dimension $2g$. The map is generically finite because the map from $\Sym^g C_t$ to $J(C_t)$ is birational. Hence, the fiber dimension in this case is zero, so the argument above does not work.

\section{Acknowledgements and declarations} The author thanks SRM AP for hosting this project. AI (Gemini) has been used to write the proof of lemma 2.4 and grammar correction and sentence repharsing has been done by GPT 5.


\begin{thebibliography}{AAAAA}


\bibitem[Ba]{Ba}K. Banerjee, {\em Bloch's conjecture on certain surfaces of general type  with $p_g=0$ and with an involution: the Enriques case}, Indagationes Mathematicae, online 15th April 2025.

\bibitem[BTh]{BTh} K. Banerjee, {\em One dimensional algebraic cycles on non-singular cubic fourfolds  in P5}, Ph.D. thesis, University of Liverpool, 2014.

\bibitem[B]{B} R.Barlow, {\em Rational equivalence of zero cycles for some surfaces with $p_g=0$}, Invent. Math. 1985, no. 2, 303-308.

\bibitem[Bog]{Bog} F. Bogomolov, {\em Holomorphic tensors and vector bundles on projective manifolds}, Izvestiya Akademii Nauk SSSR. Seriya Matematicheskaya, 42 (6): 1227–1287.

\bibitem[Bo]{Bo} E.Bombieri, {\em Canonical models of surfaces of general type}, Publications Mathématiques de l'IHÉS (1973), Volume: 42, page 171-219.


\bibitem[CK]{CK} F.Catanese, J.Keum, {\em The bicanonical map of Fake projective planes with an automorphism}, https://arxiv.org/pdf/1801.05291.pdf, 2018.


\bibitem[BKL]{BKL}S.Bloch, A.Kas, D.Lieberman, {\em Zero cycles on surfaces with $p_g=0$}, Compositio Mathematicae, tome 33, no 2(1976), page 135-145.




\bibitem[GT]{GT} V.Guletskii, A.Tikhomirov {\em Algebraic cycles on quadric sections of cubics in $\PR^4$ under the action of symplectomorphisms}, Proc. of the Edinburgh Math. Soc. 59 (2016) 377 - 392.

\bibitem[GP]{GP} V. Guletskii, C. Pedrini, {\em Chow motive of a Godeaux surface}, Algebraic Geometry: A Volume in Memory of Paolo Francia, edited by Mauro C. Beltrametti, Fabrizio Catanese, Ciro Ciliberto, Antonio Lanteri and Claudio Pedrini, Berlin, New York: De Gruyter, 2003, pp. 179-196. https://doi.org/10.1515/9783110198072.179
\bibitem[Gul]{Gul} V.Guletskii, {\em The Bloch conjecture}, https://arxiv.org/abs/2512.12451.
%\bibitem[Gul1]{Gul1}V.Guletskii, {Bloch's conjecture for surfaces with involutions and of geometric genus zero}, arXiv:1704.04187
%\bibitem[Gul2]{Gul2} V.Guletskii, {Bloch's conjecture for the surface of Craighero and Gattazzo }, arXiv:1609.04074

\bibitem[HK]{HK} D.Huybrechts, M.Kemeny, {\em Stable maps and Chow groups}, Documenta Math. 18, 2013, 507-517.
\bibitem[IM]{IM} H.Inose, M.Mizukami, {\em Rational equivalence of 0-cycles on some surfaces with $p_g=0$}, Math. Annalen, 244, 1979, no. 3, 205-217.
\bibitem[K]{K} J.Keum, {\em Quotients of fake projective planes}, Geometry, Topology, 12,  2008, 2497-2515.

\bibitem[Mi]{Mi} Y. Miyaoka, {\em On the Chern numbers of surfaces of general type}, Inventiones Mathematicae, 42 (1): 225–237.

\bibitem[M1]{M1} D. Mumford, {\em An algebraic surface with $K$ ample and $K^2=9$, $p_g=q=0$.}, American Journal Of Mathematics, Volume 101, No. 1, 1979, 233-244.


\bibitem[M]{M} D.Mumford, {\em Rational equivalence for $0$-cycles on surfaces.}, J.Math Kyoto Univ. 9, 1968, 195-204.

\bibitem[PW]{PW} C.Pedrini, C.Weibel, {\em Some examples of surfaces of general type for which Bloch's conjecture holds}, Recent Advances in Hodge Theory, London Math. Society Lecture Notes Series, 2016.


\bibitem[PY] {PY} G.Prasad, S-K.Yeung, {\em Fake projective planes}, Inventiones Mathematicae, Volume 168, 2007, 321-370.


\bibitem[R]{R} A.Roitman, {\em $\Gamma$-equivalence of zero dimensional cycles (Russian)}, Math. Sbornik. 86(128), 1971, 557-570.
\bibitem[R1]{R1}A.Roitman, {\em Rational equivalence of 0-cycles}, Math USSR Sbornik, 18, 1972, 571-588
\bibitem[R2]{R2} A.Roitman, {\em The torsion of the group of 0-cycles modulo rational equivalence}, Ann. of Math. (2), 111, 1980, no.3, 553-569

\bibitem[Voi]{Voi} C.Voisin,{\em Symplectic involutions of K$3$ surfaces act trivially on $CH_0$}, Documenta Mathematicae 17, 851-860,2012.
\bibitem[Vo]{Vo} C.Voisin, {\em Complex algebraic geometry and Hodge theory II}, Cambridge studies of Mathematics, 2002.
\bibitem[V]{V}C.Voisin, {\em Bloch's conjecture for Catanese and Barlow surfaces}, Journal of Differential Geometry, 2014, no.1, 149-175.
\bibitem[VC]{VC} C.Voisin, {\em Sur les zero cycles de certaines hypersurfaces munies d'un automorphisme}, Ann. Scuola Norm. Sup. Pisa Cl. Sci., (4), 19, 1992, no.4, 473-492.

\bibitem[Y]{Y} S.T.Yau, {\em Calabi's conjecture and some new results in algebraic geometry}, Proceedings of the National Academy of Sciences of the United States of America, 74 (5), National Academy of Sciences: 1798–1799.
\end{thebibliography}
\end{document}